\documentclass[10pt]{amsart}

\usepackage{amsmath, amsthm}
\usepackage{amssymb, mathrsfs}
\usepackage{mathtools}
\usepackage{enumerate}
\usepackage{hyperref}

\mathtoolsset{%
   mathic=true
}%

\newcommand*{\Z}{\mathbb{Z}}
\newcommand*{\R}{\mathbb{R}}
\newcommand{\SL}{\mathsf{SL}}

\newcommand*{\verts}[1]{\left\lvert #1 \right\rvert}
\newcommand*{\braces}[1]{\left\lbrace #1 \right\rbrace}
\newcommand*{\parens}[1]{\left\lparen #1 \right\rparen}

\newcommand*{\ceils}[1]{\left\lceil #1 \right\rceil}

\newcommand*{\card}{\verts}
\newcommand*{\setof}{\braces}

\newcommand*{\ceil}{\ceils}

\newcommand*{\deftobe}{\mathrel{\coloneqq}}

\newcommand*{\maps}{\colon}
\providecommand*{\st}{\,:\,}
\newcommand{\join}{\oplus}

\renewcommand*{\emptyset}{\varnothing}
\renewcommand*{\epsilon}{\varepsilon}
\renewcommand*{\phi}{\varphi}
\renewcommand*{\subset}{\subseteq}

\newcommand{\bnd}{\partial}   
\DeclareMathOperator{\cone}{cone}
\DeclareMathOperator{\conv}{conv}
\DeclareMathOperator{\lin}{lin}

\DeclareMathOperator{\pos}{pos}
\newcommand{\bigsetsum}{\bigsqcup}   
\newcommand{\pt}[1]{\mathbf{#1}}   
\newcommand{\mul}[1]{\mathbf{#1}}   

\newcommand{\intr}[1]{#1^{\circ}}   
\newcommand\commentout[1]{}
\newcommand{\J}{\mathcal{J}}
\newcommand{\K}{\mathcal{K}}
\newcommand{\C}{\mathcal{C}}

\renewcommand{\L}{\mathcal{L}}
\newcommand{\M}{\mathcal{M}}
\renewcommand{\P}{\mathcal{P}}
\newcommand{\Q}{\mathcal{Q}}
\DeclareMathOperator{\Ehr}{Ehr}
\newcommand{\lenv}{\underline{\bnd}}   
\newcommand{\llenv}{\underline{\bnd}_{\Z}}   
\newcommand{\plenv}{\underline{\bnd}^{\p}}   
\newcommand{\pllenv}{\underline{\bnd}_{\Z}^{\p}}   
\DeclareMathOperator{\den}{den}   
\newcommand{\p}{\pt{p}}
\newcommand{\x}{\pt{x}}
\renewcommand{\a}{\pt{a}}
\newcommand{\dual}[1]{#1^{\vee}}   
\newcommand{\enp}{\pt{e}_{n+1}}   
\DeclareMathOperator{\lcm}{lcm}   

\newtheorem{thm}{Theorem}[section]
\newtheorem{lem}[thm]{Lemma}
\newtheorem{cor}[thm]{Corollary}
\newtheorem{prop}[thm]{Proposition}

\theoremstyle{definition}
\newtheorem{exm}[thm]{Example}

\theoremstyle{remark}
\newtheorem{rem}[thm]{Remark}

\title{Lattice-point generating functions for free sums of convex
sets}

\author[M. Beck]{Matthias Beck}
\address[Matthias Beck]{Department of Mathematics\\
         San Francisco State University\\
         San Francisco, CA 94132\\
         USA}
\email{mattbeck@sfsu.edu}

\author[P. Jayawant]{Pallavi Jayawant}
\address[Pallavi Jayawant]{Department of Mathematics\\
              Bates College\\
              Lewiston, ME 04240\\
              USA}
\email{pjayawan@bates.edu}

\author[T.~B.~McAllister]{%
   Tyrrell B. McAllister
}%
\address[Tyrrell B. McAllister]{%
   Department of Mathematics \\
   University of Wyoming \\
   Laramie, WY 82071 \\
   USA
}%
\email{%
   tmcallis@uwyo.edu
}%

\begin{document}

\begin{abstract}
   Let $\J$ and $\K$ be convex sets in $\R^{n}$ whose affine spans
   intersect at a single rational point in $\J \cap \K$, and let
   $\J \oplus \K = \conv(\J \cup \K)$.  We give formulas for the
   generating function
   \begin{equation*}
      \sigma_{\cone(\J \oplus \K)}(z_1, \dots, z_n, z_{n+1}) =
      \sum_{(m_1, \dots, m_n) \in t(\J \oplus \K) \cap \Z^{n}}
      z_1^{ m_1 } \cdots z_n^{ m_n } z_{n+1}^{t}
   \end{equation*}
   of lattice points in all integer dilates of $\J \oplus \K$ in
   terms of $\sigma_{\cone \J}$ and $\sigma_{\cone \K}$, under
   various conditions on $\J$ and $\K$.  This work is motivated by
   (and recovers) a product formula of B.\ Braun for the Ehrhart
   series of $\P \oplus \Q$ in the case where $\P$ and $\Q$ are
   lattice polytopes containing the origin, one of which is
   reflexive.  In particular, we find necessary and sufficient
   conditions for Braun's formula and its multivariate analogue.
\end{abstract}




\thanks{\emph{Acknowledgments.} We thank B.~Braun and two
anonymous referees for their helpful remarks.  We also thank the
Rocky Mountain Mathematics Consortium for funding the 2011
graduate summer school at which this work began.  M.~Beck was
partially supported by the NSF (DMS-0810105 \& DMS-1162638).}

\maketitle


\section{Introduction}
\label{sec:Introduction}

Given arbitrary convex subsets $\J, \K \subset \R^{n}$, we denote
the convex hull of their union by $\J \join \K \deftobe \conv(\J
\cup \K)$.  We call $\J \join \K$ a \emph{free sum of $\J$ and
$\K$} when $\J$ and $\K$ each contain the origin and their
respective linear spans are orthogonal coordinate subspaces
(\emph{i.e.}, subspaces spanned by subsets of the standard basis
vectors $\pt{e}_{1}, \dotsc, \pt{e}_{n}$).\footnote{The free sum
is sometimes called the \emph{direct sum}.  Diverse conditions on
the summands appear in the literature.  Some authors require that
the origin ~\cite{BH2006}, or at least a unique point of
intersection ~\cite{McM1976, PS1967}, be in the interior of each
summand.  Others require no intersection, insisting only that the
linear spans of the summands be orthogonal coordinate subspaces
~\cite{braunreflexive, HRGZ1997}.  We require each summand to
contain the origin, but we allow the origin to be on the
boundary.} More generally, we will write ``$\J \oplus \K$ is a
free sum'' when $\J \join \K$ is a free sum of $\J$ and $\K$ up to
the action of $\SL_{n}(\Z)$ on $\R^{n}$.  A familiar example is
the octahedron $\conv \setof{\pm \pt{e}_{1}, \pm \pt{e}_{2}, \pm
\pt{e}_{3}}$ in $\R^{3}$, which is the free sum of the ``diamond''
$\conv \setof{\pm \pt{e}_{1}, \pm \pt{e}_{2}}$ and the line
segment $\conv \setof{\pm \pt{e}_{3}}$.  Free sums arise naturally in toric geometry because the
free-sum
operation is dual to the Cartesian product operation under polar
duality: $\dual{(\P \times \Q)} = \dual{\P} \oplus \dual{\Q}$.
For example, the free-sum decomposition above of the octahedron
corresponds to the decomposition of the toric variety
$\mathbb{P}^{1} \times \mathbb{P}^{1} \times \mathbb{P}^{1}$ as
the product of $\mathbb{P}^{1} \times \mathbb{P}^{1}$ and
$\mathbb{P}^{1}$.

Our goal is to understand the integer lattice points in a free sum
and its integer dilates in terms of the corresponding data for its
summands.  Of particular interest is the case of a free sum $\P
\oplus \Q$ in which $\P$ and $\Q$ are rational polytopes.  A
\emph{rational} (respectively, \emph{lattice}) \emph{polytope} in
$\R^{n}$ is a polytope all of whose vertices are in $\mathbb{Q}^n$
(respectively, the \emph{integer lattice} $\Z^{n}$).  Given a
rational polytope $\P \subset \R^{n}$, its \emph{Ehrhart series}
\[
  \Ehr_\P(t) \deftobe 1 + \sum_{k \in \Z_{\ge 1}} \card{k \P \cap
  \Z^n} t^k
\]
is the generating function of the \emph{Ehrhart quasi-polynomial}
of $\P$, which counts the integer lattice points in $k\P$ as a
function of an integer dilation parameter $k$.  Let $\den \P$
denote the \emph{denominator} of $\P$, the smallest positive
integer such that the corresponding dilate of $\P$ is a lattice
polytope.  A famous theorem of Ehrhart ~\cite{ehrhartpolynomial}
says that
\[
   \Ehr_{\P}(t) = \frac{ \delta_\P(t) }{ (1 - t^{\den \P})^{\dim
   \P + 1} }
\]
for some polynomial $\delta_\P$, the \emph{$\delta$-polynomial} of
$\P$.  (Common alternative names for the $\delta$-polynomial
include \emph{$h^*$-polynomial} and \emph{Ehrhart $h$-vector}.)
See, e.g., ~\cite{ccd,hibi,stanleyec1} for this and many more
facts about Ehrhart series.

Our work is motivated by the following result of B.\ Braun, which
expresses the $\delta$-polynomial of $\P \oplus \Q$ in terms of
the $\delta$-polynomials of $\P$ and $\Q$ when $\P$ is a reflexive
polytope (defined in Section ~\ref{sec:BraunFormula} below).
\begin{thm}[\cite{braunreflexive}]\label{Braunthm}
   Suppose that $\P, \Q \subset \R^n$ are lattice polytopes such
   that $\P$ is reflexive, $\Q$ contains the origin in its
   relative interior, and $\P \oplus \Q$ is a free sum.  Then
   \begin{equation}
      \label{eq:braunreflexive}
      \delta_{\P \oplus \Q} = \delta_{\P} \, \delta_{\Q} \, .
   \end{equation}
   That is, in terms of Ehrhart series,
   \begin{equation}
      \Ehr_{\P \oplus \Q}(t) = (1-t)\Ehr_{\P}(t) \Ehr_{\Q}(t).
      \label{eq:braunreflexiveseries}
   \end{equation}
\end{thm}

Our first main result, Theorem
~\ref{thm:MultivariateBraunsFormula} below, gives a multivariate
generalization of Theorem ~\ref{Braunthm} for arbitrary compact
convex sets.  Our second main result, Theorem
~\ref{thm:ConverseBraunFormula} below, characterizes the free sums
of rational polytopes that satisfy our multivariate generalization
of equation ~\eqref{eq:braunreflexiveseries}.  A characterization
of the free sums satisfying equation
~\eqref{eq:braunreflexiveseries} itself is a consequence.  Before
stating our results, we first need to define some notation.

The Ehrhart series is a specialization of a multivariate Laurent
series defined as follows.  Let $\alpha \maps \R^{n} \to \R^{n+1}$
be the affine embedding $(a_{1}, \dotsc, a_{n}) \mapsto (a_{1},
\dotsc, a_{n}, 1)$.  Given a convex set $\K \subset \R^{n}$, let
$\cone \K \subset \R^{n+1}$ be the set of all nonnegative scalar
multiples of elements of $\alpha(\K)$.  Equivalently, $\cone \K$
is the intersection of all linear cones containing $\alpha(\K)$.
Write $S_{\Z}$ for the set of integer lattice points in a set $S$.
The \emph{lattice-point generating function} $\sigma_{S}(\mul{z})$
of $S \subset \R^{n+1}$ is the formal multivariate Laurent series
\begin{equation*}
   \sigma_{S}(\mul{z}) \deftobe \sum_{\pt{m} \in S_{\Z}}
   \mul{z}^{\pt{m}}.
\end{equation*}
(Here we follow the convention of writing $\sigma_{S}(\mul{z})$
for $\sigma_{S}(z_{1},\dotsc,z_{n+1})$ and $\mul{z}^{\pt{m}}$ for
$z_{1}^{m_{1}} \dotsm z_{n+1}^{m_{n+1}}$, where
$\pt{m}=(m_{1},\dotsc,m_{n+1})$.)  The Ehrhart series
$\Ehr_{\P}(t)$ then arises as a specialization of $\sigma_{\cone
\P}(\mul{z})$:
\begin{equation*}
   \Ehr_{\P}(t) = \sigma_{\cone \P} (1, \dotsc, 1, t) \, .
\end{equation*}

Let $\pt{e}_{1}, \dotsc, \pt{e}_{n}, \pt{e}_{n+1}$ denote the
standard basis vectors in $\R^{n+1}$.  Given a closed linear cone
$\C \subset \R^{n+1}$ not containing $-\enp$, define the
projection $\epsilon_{\C} \maps \C \to \bnd \C$ (where ``$\bnd$''
denotes relative boundary) by letting
\[\epsilon_{\C}(\x) \deftobe \x -
\max\setof{\lambda \in \R \st \x - \lambda \enp \in \C} \enp.\]
Given a compact convex set $\J\subset \R^{n}$, we write
$\epsilon_{\J}$ as an abbreviation for $\epsilon_{\cone \J}$.  (We
require $\J$ to be compact so that $\cone \J$ is closed.)  The
\emph{lower envelope} of $\C$ is
\begin{equation*}
   \lenv \C \deftobe \epsilon_{\C}(\C) \, .
\end{equation*}
Thus, the lower envelope of $\C$ is the set of points that are
``vertically minimal'' within $\C$.  The \emph{lower
lattice envelope} of $\C$ is
\[
  \llenv \C \deftobe \epsilon_{\C}(\C_{\Z}) \, .
\]
Thus, the lower lattice envelope is the vertical projection of the
lattice points in $\C$ onto the lower envelope of $\C$.  Observe
that the lower lattice envelope is \emph{not} necessarily the set
$(\lenv \C)_{\Z}$ of lattice points in the lower envelope of $\C$.
In general, some elements of $\llenv \C$ may not be lattice
points.

\begin{thm}[proved on p.\ \pageref{proof:MultivariateBraunsFormula}]
   \label{thm:MultivariateBraunsFormula}
   Suppose that $\J, \K \subset \R^{n}$ are convex sets such that
   $\J$ is compact and $\J \oplus \K$ is a free sum.  Suppose
   moreover that $\llenv \cone \J = (\lenv \cone \J)_{\Z}$.  Then
   \begin{equation}
      \label{eq:MultivariateBraunsEquation}
      \sigma_{\cone (\J \oplus \K)}(\mul{z}) = (1 - z_{n+1}) \,
      \sigma_{\cone \J}(\mul{z}) \, \sigma_{\cone \K}(\mul{z}) \, .
   \end{equation}
\end{thm}
We call equation \eqref{eq:MultivariateBraunsEquation} the
\emph{multivariate Braun equation}.  Our second main result states
that, when $\J$ and $\K$ are rational polytopes, the converse of
Theorem ~\ref{thm:MultivariateBraunsFormula} also holds.  (Whether
the converse holds for free sums $\J \oplus \K$ of arbitrary
convex sets is still an open question.)  Given a rational polytope
$\P$ containing the origin, we observe in Proposition
~\ref{ratflexiveprop} below that $\llenv \cone \P = (\lenv \cone
\P)_{\Z}$ if and only if the polar dual $\dual \P$ of $\P$
(relative to its linear span) is a lattice polyhedron.  We show
that, if a free sum of rational polytopes satisfies the
multivariate Braun equation, then the dual of one of those
polytopes is a lattice polyhedron.

\begin{thm}[proved on p.\ \pageref{proof:ConverseBraunFormula}]
   \label{thm:ConverseBraunFormula}
   Let $\P, \Q \subset \R^{n}$ be rational polytopes such that $\P
   \oplus \Q$ is a free sum.  Then
   \begin{equation}
      \label{eq:BraunsFormulaHolds}
      \sigma_{\cone (\P \oplus \Q)}(\mul{z}) = (1 -
      z_{n+1})\sigma_{\cone \P}(\mul{z}) \sigma_{\cone
      \Q}(\mul{z})
   \end{equation}
   if and only if either $\dual\P$ or $\dual\Q$ is a lattice
   polyhedron.
\end{thm}

The univariate analogue of Theorem \ref{thm:ConverseBraunFormula}
is a consequence:

\begin{thm}[proved on p.\ \pageref{proof:ConverseBraunUnivariate}]
   \label{thm:ConverseBraunUnivariate}
   Let $\P, \Q \subset \R^{n}$ be rational polytopes such that $\P
   \oplus \Q$ is a free sum.  If either $\dual \P$ or $\dual \Q$
   is a lattice polyhedron, then
   \begin{equation}
      \label{eq:ConverseBraunEhrhart}
      \Ehr_{\P \oplus \Q}(t) = (1-t)\Ehr_{\P}(t) \Ehr_{\Q}(t)
   \end{equation}
   and hence
   \begin{equation}
      \label{eq:ConverseBraunDelta}
      \delta_{\P \oplus \Q}(t) =
      \frac{(1-t)(1-t^{\lcm(\den\P,\den\Q)})^{\dim\P + \dim \Q +
      1}}{(1-t^{\den \P})^{\dim\P + 1}(1 - t^{\den\Q})^{\dim\Q +
      1}}\,\delta_{\P}(t) \, \delta_{\Q}(t).
   \end{equation}
   Conversely, if either equation \eqref{eq:ConverseBraunEhrhart}
   or equation \eqref{eq:ConverseBraunDelta} holds, then either
   $\dual \P$ or $\dual \Q$ is a lattice polyhedron.  In
   particular, if $\P, \Q \subset \R^{n}$ are lattice polytopes
   such that $\P \oplus \Q$ is a free sum, then
   \begin{equation}
      \label{eq:ConverseBraunDeltaLattice}
      \delta_{\P \oplus \Q} = \delta_{\P}\, \delta_{\Q}
   \end{equation}
   if and only if either $\dual \P$ or $\dual \Q$ is a lattice
   polyhedron.
\end{thm}

After laying the groundwork for our approach to free sums in
Section ~\ref{sec:FreeSumDecompositions}, we prove Theorem
~\ref{thm:MultivariateBraunsFormula} and various corollaries,
including Theorem ~\ref{Braunthm}, in Section
~\ref{sec:BraunFormula}.  In Section ~\ref{NonBraunFormulaCase},
we give an expression for $\sigma_{\cone(\P \oplus \K)}$ when $\P
\oplus \K$ is an arbitrary free sum in which $\P$ is a rational
polytope (Theorem~\ref{thm:NonBraunFormulaCase}).  We then use
this expression to prove Theorems~\ref{thm:ConverseBraunFormula}
and~\ref{thm:ConverseBraunUnivariate}.

Although Sections ~\ref{sec:BraunFormula} and
~\ref{NonBraunFormulaCase} address only the case where $\J \join
\K$ is a free sum, our approach is not confined to this situation.
Section ~\ref{sec:nonlatticepoints} introduces the notion of
\emph{affine free sums} $\J \oplus \K$, where $\J$ and $\K$ may
intersect at an arbitrary rational point.  We derive formulas for
the lattice-point generating functions of cones over affine free
sums $\J \oplus \K$ under certain conditions on $\J$ and $\K$.
One case of interest that satisfies these conditions is an affine
free sum $\P \oplus \K$ where $\P$ is a Gorenstein polytope of
index $k$ intersecting an orthogonal convex set $\K$ at the unique
point $\p \in \P$ such that $k\p$ is a lattice point in the
relative interior of $k\P$ (Corollary
~\ref{cor:GorensteinBraunCor}).

\section{Decompositions of cones over free sums}
\label{sec:FreeSumDecompositions}

We begin our study of the generating function $\sigma_{\cone(\J
\oplus \K)}$ from the vantage point of the following easy
identity: Given any convex sets $\J, \K \subset \R^{n}$, the
convex hull $\J \join \K$ of their union satisfies
\begin{equation}\label{freesumminkowskyidentity}
   \cone(\J \join \K) = \cone \J + \cone \K \, ,
\end{equation}
where the sum on the right is the Minkowski sum $S + T \deftobe
\setof{s + t \st s \in S, t \in T}$.
The goal of this section is to provide two refinements to equation
\eqref{freesumminkowskyidentity}, first by making the equation
``disjoint'', and then by restricting the equation to lattice
points.  As it stands, equation \eqref{freesumminkowskyidentity}
``double counts'' elements of $\cone(\J \join \K)$, in the sense
that there are many ways to express an element of the left-hand
side as a sum from the right-hand side.  Proposition
~\ref{prop:FreeSumIntoDisjointBoundaryandCone} below gives a
non-double-counting version of
equation~\eqref{freesumminkowskyidentity} under certain
conditions on $\J$ and $\K$. %
Proposition ~\ref{prop:LatticeFreeSumIntoDisjointBoundaryandCone}
below provides a similar expression for the integer lattice points
in $\cone(\J \oplus \K)$ when $\J \oplus \K$ is a free sum.

First we make a few additional notational remarks: We write $\pi
\maps \R^{n+1} \to \R^{n}$ for the orthogonal projection
\begin{equation*}
   \pi \maps (x_{1}, \dotsc, x_{n}, x_{n+1}) \mapsto (x_{1},
   \dotsc, x_{n}) \, .
\end{equation*}
Given a subset $S$ of $\R^{n}$ or $\R^{n+1}$, let $\lin S$ be the
linear span of $S$.  We say that two sublattices $\L, \M \subset
\Z^{n}$ are \emph{complementary sublattices of $\Z^{n}$} if each
element of $(\lin(\L \cup \M))_{\Z}$ is the sum of a unique
element of $\L$ and a unique element of $\M$.  Hence, when $\J
\oplus \K$ is a free sum, $(\lin \J)_{\Z}$ and $(\lin \K)_{\Z}$
are complementary sublattices of $\Z^{n}$.

%

Equation \eqref{freesumminkowskyidentity} says that
\begin{equation*}
\label{eq:FreeSumIntoUnionOfCones}
   \cone(\J \join \K) = \bigcup_{\pt{x} \in \cone \J} \left(\pt{x}
   + \cone \K \right).
\end{equation*}
Using the concept of the lower envelope (defined in Section
~\ref{sec:Introduction}), we can replace the union above by a
disjoint union.  This yields the desired ``disjoint'' version of
equation \eqref{freesumminkowskyidentity}.  We use $\bigsetsum$ to
denote disjoint union.

\begin{prop}\label{prop:somewhatintermediate}
   Suppose that $\J, \K \subset \R^{n}$ are convex sets with $\J$
   compact and $\pt{0} \in \K$.  Suppose in addition that the
   linear spans of $\J$ and $\K$ intersect trivially.  Then
   \begin{equation*}
      \cone(\J\oplus \K) = \bigsetsum_{\pt{x} \in \lenv \cone \J}
      (\pt{x} + \cone \K) \, .
   \end{equation*}
   \label{prop:FreeSumIntoDisjointBoundaryandCone}
\end{prop}

\begin{proof}
   We first show that the union on the right-hand side is a
   disjoint union.  Suppose that
   \begin{equation*}
      \pt{x}_{1} + \pt{y}_{1} =
      \pt{x}_{2} + \pt{y}_{2}
   \end{equation*}
   for some $\pt{x}_{1}, \pt{x}_{2} \in \lenv \cone \J$ and
   $\pt{y}_{1}, \pt{y}_{2} \in \cone \K$.  Then we have
   $\pi(\pt{x}_{1}) + \pi(\pt{y}_{1}) = \pi(\pt{x}_{2}) +
   \pi(\pt{y}_{2})$.  Hence
   \begin{equation*}
      \pi(\pt{x}_{1}) - \pi(\pt{x}_{2}) =
      \pi(\pt{y}_{2})-\pi(\pt{y}_{1})\in \lin \J\cap \lin \K
   \end{equation*}
   because the left-hand side of the equality is in $\lin \J$
   while the right-hand side is in $\lin \K$.  Since $\lin \J \cap
   \lin \K = \setof{\pt{0}}$, it follows that $\pi(\pt{x}_{1}) =
   \pi(\pt{x}_{2})$.  Now, the preimage
   $\pi^{-1}(\pi(\pt{x}_{1}))$ contains exactly one point in
   $\lenv \cone \J$, so $\pt{x}_{1} = \pt{x}_{2}$, proving
   disjointness.
   
   It remains only to show that
   \begin{equation*}
      \bigsetsum_{\pt{x} \in \lenv \cone \J} (\pt{x} + \cone \K) =
      \bigcup_{\pt{x} \in \cone \J} \left(\pt{x} + \cone \K
      \right).
   \end{equation*}
   The left-hand side is contained in the right-hand side because
   $\J$ is compact, so $\lenv \cone \J \subset \cone \J$.
   Conversely, if $\pt{w} \in \pt{x} + \cone \K$ for some $\pt{x}
   \in \cone \J$, then
   \begin{equation}
   \label{eq:DropToEnvelope}
      \pt{w} - (\pt{x} - \epsilon_\J(\pt{x})) \in
      \epsilon_\J(\pt{x}) + \cone \K \, .
   \end{equation}
   Now, $\pt{x} - \epsilon_\J(\pt{x})$ is a nonnegative multiple
   of $\enp$, which is in $\cone \K$ because $\pt{0} \in \K$.
   Thus, adding $\pt{x} - \epsilon_\J(\pt{x})$ to both sides of
   \eqref{eq:DropToEnvelope} yields $\pt{w} \in
   \epsilon_\J(\pt{x}) + \cone \K$.  Since $\epsilon_\J(\pt{x})
   \in \lenv \cone \J$, this proves the claim.
\end{proof}

Our ultimate goal is to understand the generating function
$\sigma_{\cone(\J \oplus \K)}$, so we need a version of the
disjoint union in Proposition
~\ref{prop:FreeSumIntoDisjointBoundaryandCone} that is restricted
to the lattice points in $\cone(\J \oplus \K)$.  This is provided
by the following proposition.  See also Figure
~\ref{fig:LatticeFreeSumIntoDisjointBoundaryandCone}.

\begin{prop}
\label{prop:LatticeFreeSumIntoDisjointBoundaryandCone}
   Suppose that $\J, \K \subset \R^{n}$ are convex sets such that
   $\J$ is compact and $\J \oplus \K$ is a free sum.  Then
   \begin{equation*}
      \cone(\J\oplus \K)_{\Z} = \bigsetsum_{\pt{x} \in \llenv
      \cone \J} (\pt{x} + \cone \K)_{\Z} \, .
   \end{equation*}
\end{prop}

\begin{figure}[tbp]
   \centering
   \includegraphics{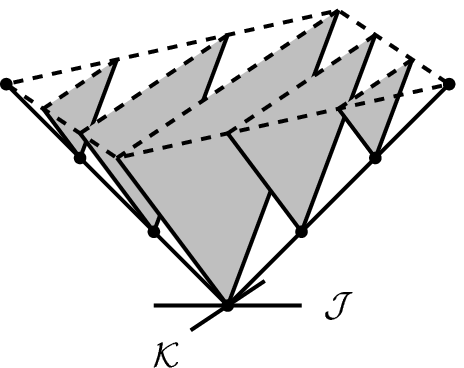}
   \caption{A depiction of $\cone(\J \oplus \K)$.  The dots
   indicate elements of $\llenv \cone \J$.  The shaded regions
   represent translations of $\cone \K$ by elements of $\llenv
   \cone \J$.  The import of
   Proposition ~\ref{prop:LatticeFreeSumIntoDisjointBoundaryandCone}
   is that all lattice points in $\cone(\J \oplus \K)$ are within
   these shaded regions.}
   \label{fig:LatticeFreeSumIntoDisjointBoundaryandCone}
\end{figure}

\begin{proof}
   The elements of the right-hand side are lattice points that are
   contained in $\cone(\J \oplus \K)$ by the previous proposition.
   Hence, such elements are in the left-hand side.
   
   To prove the converse containment, let $\pt{w} \in \cone(\J
   \oplus \K)_{\Z}$ be given.  By Proposition
   ~\ref{prop:somewhatintermediate}, there exist $\pt{x} \in \lenv
   \cone \J$ and $\pt{y} \in \cone \K$ such that $\pt{w} = \pt{x}
   + \pt{y}$.  Thus, $\pi(\pt{w}) = \pi(\pt{x}) + \pi(\pt{y})$.
   Now, $\pi(\pt{w})$ is an integer lattice point in $\lin(\J \cup
   \K)$, while $\pi(\pt{x}) \in \lin \J$ and $\pi(\pt{y}) \in \lin
   \K$.  Since $(\lin \J)_{\Z}$ and $(\lin \K)_{\Z}$ are
   complementary sublattices of $\Z^{n}$, it follows that
   $\pi(\pt{x}) \in \Z^{n}$.  Furthermore, $\pt{x} \in \cone \J$
   and $\pt{0} \in \J$, so there exists an integer $\lambda$ such
   that $\pi(\pt{x})+\lambda\pt{e}_{n+1} \in (\cone \J)_{\Z}$.
   Therefore, $\x = \epsilon_\J(\pi(\pt{x})+\lambda\pt{e}_{n+1})
   \in \llenv \cone \J$.
\end{proof}

\begin{rem}
   If $\J \oplus \K$ is not a free sum, then Proposition
   \ref{prop:LatticeFreeSumIntoDisjointBoundaryandCone} does not
   hold.  For example, let $\J \subset \R^{2}$ be the segment
   $[(-1,0), (1,0)]$, and let $\K \subset \R^{2}$ be the segment
   $[(-1,-2), (1,2)]$.  Note that $(\lin \J)_{\Z}$ and $(\lin
   \K)_{\Z}$ are not complementary sublattices in $\Z^{2}$, so $\J
   \oplus \K$ is not a free sum.  The equation in Proposition
   \ref{prop:LatticeFreeSumIntoDisjointBoundaryandCone} fails to
   hold in this case because, for example, the lattice point
   $(1,1,1)$ appears in $\cone(\J\oplus \K)_{\Z}$ but not in
   $\bigsetsum_{\pt{x} \in \llenv \cone \J} (\pt{x} + \cone
   \K)_{\Z}$.
\end{rem}

\section{Sufficient conditions for the multivariate Braun equation}
\label{sec:BraunFormula}

The multivariate Braun equation
\eqref{eq:MultivariateBraunsEquation} does not hold for all free
sums $\J \oplus \K$ of convex sets.  In this section, we give
conditions on $\J$ and $\K$ that suffice to imply equation
\eqref{eq:MultivariateBraunsEquation}.  The conditions we give
generalize those originally given by Braun in
~\cite{braunreflexive}.  In the next section, we will show that,
conversely, our conditions are necessary in the case where $\J$
and $\K$ are rational polytopes.

To apply
Proposition~\ref{prop:LatticeFreeSumIntoDisjointBoundaryandCone},
we need to get our hands on the set $\llenv \cone \J$.  The next
proposition considers the case where all the elements of this set
are integer lattice points.

\begin{prop}\label{settogeneratingfctprop}
   Let $\J \subset \R^{n}$ be a compact convex set containing the
   origin.  Then the following conditions are equivalent:
   \begin{enumerate}[{\rm (a)}]
      \item  
      $\llenv \cone \J = (\lenv \cone \J)_{\Z}$,
   
      \item  
      $(\lenv \cone \J)_{\Z} = (\cone \J)_{\Z} \setminus (\cone \J
      + \enp)_{\Z}$,
   
      \item  
      $\sigma_{\lenv \cone \J}(\mul{z}) = (1 - z_{n+1}) \,
      \sigma_{\cone \J}(\mul{z})$.
   \end{enumerate}
\end{prop}

\begin{proof}
   We start by proving that (a) and (b) are equivalent.  First,
   note that the set containments
   \[
      \llenv \cone \J \supseteq (\lenv \cone \J)_{\Z}
   \]
   and
   \[
      (\lenv \cone \J)_{\Z} \subseteq (\cone \J)_{\Z} \setminus
      (\cone \J + \enp)_{\Z} 
   \]
   always hold.  To see that the respective converse containments
   are equivalent, observe that $\x \mapsto \epsilon_{\J}(\x)$ is
   a bijection between non-lower-envelope points in $(\cone
   \J)_{\Z} \setminus (\cone \J + \enp)_{\Z}$ and non-lattice
   points in $\llenv \cone \J$, with inverse bijection $(\a,
   \lambda) \mapsto (\a, \ceil{\lambda})$.  Thus, if either
   containment above is an equality, then so too is the other.
   
   Finally, the left- (resp.\ right-) hand side of (c) lists the
   points of the left- (resp.\ right-) hand side of (b) in
   generating-function form, so (b) and (c) are equivalent.
\end{proof}

Theorem ~\ref{thm:MultivariateBraunsFormula} is now an easy
corollary of the previous proposition.

\begin{proof}[Proof of Theorem ~\ref{thm:MultivariateBraunsFormula}
   (stated on p.~\pageref{thm:MultivariateBraunsFormula})]
   \label{proof:MultivariateBraunsFormula}
   Since $\llenv \cone \J = (\lenv \cone \J)_{\Z},$ the
   set-the\-o\-re\-tic equation in Proposition
   ~\ref{prop:LatticeFreeSumIntoDisjointBoundaryandCone} can be
   restated in terms of generating functions as follows:
   \[\sigma_{\cone (\J \oplus \K)}(\mul{z}) = \sigma_{ \lenv \cone
   \J } (\mul{z}) \, \sigma_{ \cone \K } (\mul{z}) \, .
   \]
   The theorem now follows from the equivalence of (a) and (c) in
   Proposition ~\ref{settogeneratingfctprop}.
\end{proof}

%

The conditions in Proposition ~\ref{settogeneratingfctprop} take on
an especially nice form when the convex set $\J$ is a rational
polytope.  We now show that, in this case, these conditions are
equivalent to the condition that the polar dual of $\J$ is a
lattice polyhedron.  We recall the relevant definitions.

The \emph{\textup{(}polar\textup{)} dual} of a polytope $\P
\subset \R^n$ containing the origin is defined to be the
polyhedron
\begin{equation*}
  \P^\vee \deftobe \setof{\phi \in (\lin \P)^{\ast} \st
  \text{$\phi(\a) \le 1$ for all $\a \in \P$}},
\end{equation*}
where $V^{\ast}$ denotes the set of all real-valued linear
functionals on a vector space $V$.  (Note that we use $\P^\vee$ to
refer to the dual of $\P$ \emph{with respect to the linear span of
$\P$}.)  In general, $\P^{\vee}$ may be unbounded, but, if $\pt{0}
\in \intr{\P}$, then $\P^\vee$ is a polytope.  (Here we write
$\intr{S}$ for the interior of a set $S$ relative to the subspace
topology on $\lin S$.)  Let $\phi_{1},\dotsc, \phi_{k},
\psi_{1},\dotsc,\psi_{\ell} \in (\lin \P)^{\ast}$ be linear
functionals such that
\begin{equation*}
   \P = \setof{\a \in \lin \P \st \text{$\phi_{1}(\a), \dotsc,
   \phi_{k}(\a) \le 1$ and $\psi_{1}(\a), \dotsc, \psi_{\ell}(\a)
   \le 0$}}.
\end{equation*}
Then $\dual \P$ can be expressed as the Minkowski sum of a
polytope and a polyhedral cone in the dual space $(\lin
\P)^{\ast}$ as follows:
\begin{equation*}
   \dual \P = \conv\setof{\phi_{1}, \dotsc, \phi_{k}} + 
   \pos\setof{\psi_{1}, \dotsc, \psi_{\ell}},
\end{equation*}
where $\pos S$ denotes the \emph{positive hull} $\setof{\lambda \a
\st \text{$\a \in \conv S$ and $\lambda \ge 0$}}$ of a set $S$.
We call $\P^\vee$ a \emph{lattice polyhedron} if its vertices are
in the \emph{dual integer lattice} defined by
\[
   (\lin \P)_{\Z}^{\ast} \deftobe \setof{\phi \in (\lin \P)^{\ast}
   \st \phi(\a) \in \Z \text{ for all } \a \in (\lin \P)_\Z}.
\]
%
   A polytope $\P$ is \emph{reflexive} if both $\P$ and
   $\P^{\vee}$ are lattice polytopes.
Reflexive polytopes were introduced by Victor Batyrev to study
mirror symmetry in string theory ~\cite{batyrevdualpoly}.

Hibi ~\cite{hibidual} showed that a lattice polytope $\P$
containing the origin in its interior is reflexive if and only if
$\left( k\P \setminus (k-1)\P \right)_\Z = \left( \bnd (k\P)
\right)_\Z$ for all integers $k \ge 2$.  This latter condition, in
turn, is equivalent to $\llenv \cone \P = (\lenv \cone \P)_{\Z}$.
Hibi's proofs carry over with virtually no change if we merely
assume that $\P$ is rational and contains the origin (not
necessarily in its interior).  Hibi's arguments then show that
$\dual \P$ is a lattice polyhedron if and only if $\llenv \cone \P
= (\lenv \cone \P)_{\Z}$.  We include a proof of this equivalence
for completeness (Proposition ~\ref{ratflexiveprop} below).
Non-lattice rational polytopes with lattice duals have appeared,
e.g., in ~\cite{fisetkasprzyk}, which gives a rational analogue of
a theorem of Hibi on the Ehrhart series of reflexive polytopes
~\cite{hibidual}.

\begin{prop}\label{ratflexiveprop}
   Let $\P$ be a rational polytope with $\pt{0} \in \P$.  Then
   $\dual \P$ is a lattice polyhedron if and only if $\llenv \cone
   \P = (\lenv \cone \P)_{\Z}$.
\end{prop}

\begin{proof}
   Suppose that $\dual \P$ is a lattice polyhedron.  It is clear
   that $\llenv \cone \P \supseteq (\lenv \cone \P)_{\Z}$.  To
   prove the converse containment, let $\x \in \llenv \cone \P$ be
   given.  By definition of the lower lattice envelope, we have
   that $\pi(\x) \in \Z^{n}$.  Let $\phi_1, \phi_2, \dots, \phi_k$
   be the vertices of $\P^\vee$, and let $\lambda \deftobe
   \max\setof{\phi_{1}(\pi(\x)),\dotsc,\phi_{k}(\pi(\x))}$.  Then
   $\pi(\x) \in \lambda \P$ while $\pi(\x) \notin (\lambda -
   \epsilon) \P$ for $0 < \epsilon < \lambda$.  Thus, $\x =
   (\pi(\x), \lambda) \in \lenv \cone \P$.  Furthermore, since
   each $\phi_{i}$ is a dual integer lattice point, we have that
   $\lambda \in \Z$, which implies that $\x \in (\lenv \cone
   \P)_\Z$, proving the desired containment.
   
   Conversely, suppose that $\dual \P$ has a vertex $\phi_{j}
   \notin (\lin \P)_{\Z}^{*}$.  Let $\Lambda \subset (\lin
   \P)_{\Z}$ be the sublattice of $(\lin \P)_{\Z}$ on which
   $\phi_{j}$ evaluates as an integer.  Thus, $\Lambda$ is a
   full-rank proper sublattice of $(\lin \P)_{\Z}$.  Let $F$ be
   the facet of $\P$ supported by the hyperplane $\phi_{j} = 1$.
   Then there exists a lattice point $\a \in (\pos F)_{\Z}
   \setminus \Lambda$.  (This may be seen by observing that $\pos
   F$ is a full-dimensional cone containing some element of
   $\Lambda$ in its interior.  Hence, $\pos F$ contains some
   $\Lambda$-translate of a fundamental domain of $\Lambda$, which
   in turn contains elements of $\Z^{n} \setminus \Lambda$.%
   ) 
   We then have that $\phi_{j}(\a) \notin \Z$ but
   $(\a, \phi_{j}(\a)) \in \llenv \cone \P$, so that $\llenv \cone
   \P \not\subset (\lenv \cone \P)_{\Z}$.
\end{proof}

As a corollary of Propositions \ref{settogeneratingfctprop} and
\ref{ratflexiveprop}, we find that the multivariate Braun equation
\eqref{eq:MultivariateBraunsEquation} holds when one of the
summands is a rational polytope whose polar dual is a lattice
polyhedron.

\begin{cor}\label{rationalbraunmultivarcor}
   Let $\P \subset \R^n$ be a rational polytope such that $\dual
   \P$ is a lattice polyhedron, and let $\K \subset \R^{n}$ be a
   convex set such that $\P \oplus \K$ is a free sum.  Then
   \begin{equation*}
      \sigma_{\cone (\P \oplus \K)}(\mul{z}) = (1 - z_{n+1}) \,
      \sigma_{\cone \P}(\mul{z}) \, \sigma_{\cone \K}(\mul{z}) \,.
   \end{equation*}
\end{cor}

By applying the specialization $\Ehr_{\P}(t) = \sigma_{\cone
\P}(1,\dotsc,1,t)$, we arrive at the following generalization of
Braun's Theorem ~\ref{Braunthm}.

\begin{cor}\label{braunthmgeneralizedcor}
   If $\P, \Q \subset \R^{n}$ are rational polytopes such that
   $\dual \P$ is a lattice polyhedron and $\P \oplus \Q$ is a free
   sum, then
   \begin{equation*}
      \Ehr_{\P \oplus \Q}(t) = (1-t)\Ehr_{\P}(t) \Ehr_{\Q}(t)
   \end{equation*}
   and hence
   \begin{equation*}
      \delta_{\P \oplus \Q}(t) =
      \frac{(1-t)(1-t^{\lcm(\den\P,\den\Q)})^{\dim\P + \dim \Q +
      1}}{(1-t^{\den \P})^{\dim\P + 1}(1 - t^{\den\Q})^{\dim\Q +
      1}}\,\delta_{\P}(t) \, \delta_{\Q}(t).
   \end{equation*}
   In particular, if $\P, \Q \subset \R^{n}$ are lattice polytopes
   such that $\dual \P$ is a lattice polyhedron and $\P \oplus \Q$
   is a free sum, then
   \begin{equation*}
      \delta_{\P \oplus \Q} = \delta_{\P}\, \delta_{\Q}
      \, .
   \end{equation*}
\end{cor}


\begin{rem}
   Corollary ~\ref{braunthmgeneralizedcor} recovers the following
   generalization of Theorem ~\ref{Braunthm}, due to
   Braun~\cite[Corollary 1]{braunreflexive}: Let $\P$ and $\Q$ be
   as in Theorem ~\ref{Braunthm}, and let $\P'$ (respectively,
   $\Q'$) be a lattice polytope equal to the intersection of $\P$
   (resp., $\Q$) with a finite collection of half-spaces in $\lin
   \P$ (resp., $\lin \Q$) bounded by hyperplanes passing through
   the origin.  Then $\delta_{\P' \oplus \Q'} = \delta_{\P'} \,
   \delta_{\Q'}$.
   
   There are lattice polytopes covered by our
   Corollary~\ref{braunthmgeneralizedcor} that do not satisfy the
   conditions of Braun's ~\cite[Corollary 1]{braunreflexive}.  For
   example, let $\P \subset \R^{2}$ be the polygon $\conv
   \setof{(-1, 0), (1, 0), (3, 1), (-3, 1)}$.  Then $\P$ is not
   contained in any reflexive polygon, but the dual of $\P$ is a
   lattice polyhedron.  (The polygon $\P$ is a $2$-dimensional
   analogue of a so-called \emph{top} polytope.  Top polytopes,
   like reflexive polytopes, originally arose in string theory~
   \cite{CanSka1997}.)
   
\end{rem}

\section{Necessary conditions for the multivariate Braun equation}
\label{NonBraunFormulaCase}

In this section, we prove Theorem ~\ref{thm:ConverseBraunFormula},
the converse of Theorem ~\ref{thm:MultivariateBraunsFormula} in
the case where the summands are rational polytopes.  That is, we
show that, if $\P$ and $\Q$ are rational polytopes containing the
origin such that
\begin{equation*}
   \sigma_{\cone (\P \oplus \Q)}(\mul{z}) = (1 - z_{n+1}) \,
   \sigma_{\cone \P}(\mul{z}) \, \sigma_{\cone \Q}(\mul{z}),
\end{equation*}
then either $\dual \P$ or $\dual \Q$ is a lattice polyhedron.  We
also prove Theorem~\ref{thm:ConverseBraunUnivariate}, the
univariate version of Theorem~\ref{thm:ConverseBraunFormula}.

Fix a rational polytope $\P \subset \R^{n}$ such that $\pt{0} \in
\P$, and let $\K \subset \R^{n}$ be a convex set such that
\mbox{$\P \oplus \K$} is a free sum.  As in the previous section,
we approach the gen\-er\-at\-ing function $\sigma_{\cone(\P \oplus
\K)}$ via the decomposition of $\cone(\P \oplus \K)_{\Z}$ given by
Proposition~\ref{prop:LatticeFreeSumIntoDisjointBoundaryandCone}.
The first step, therefore, is to find a useable description of the
lower lattice envelope $\llenv \cone \P$ in the case where we do
not necessarily have $\llenv \cone \P = (\lenv \cone \P)_{\Z}$.

Write $d(\P)$ for the denominator $\den(\dual{\P})$ of $\dual \P$.
For each nonnegative integer $i$, let
\begin{align*}
   \cone^{i} \P &\deftobe \cone \P + \tfrac{i}{d(\P)} \enp \, , \\
   \cone_{i} \K & \deftobe \cone \K - \tfrac{i}{d(\P)} \enp \, .
\end{align*}
(Observe that the definition of $\cone_{i}\K$ depends upon the
choice of $\P$, although this is not reflected in the notation.)
We similarly define the \emph{shifted lower envelopes}
$\smash{\lenv \cone^{i} \P \deftobe \lenv \cone \P +
\frac{i}{d(\P)} \enp}$ and $\smash{\lenv \cone_{i} \K \deftobe
\lenv \cone \K - \frac{i}{d(\P)} \enp}$ of these shifted cones.
The \emph{rind} of $\cone \P$ is $(\cone \P) \setminus (\cone \P +
\enp)$.

Proposition ~\ref{prop:RindAndStrata} below is a generalization of
Proposition ~\ref{settogeneratingfctprop} as applied to any
rational polytope containing the origin.  Before giving the formal
statement of Proposition \ref{prop:RindAndStrata}, we give an
informal summary.  See also Figure ~\ref{fig:RindAndStrata}.
\begin{itemize}
   \item 
   Each point in the lower lattice envelope is the result of
   taking a unique lattice point on some shifted lower envelope
   contained in the rind of $\cone \P$ and projecting that lattice
   point down to the lower envelope.
   
   \item 
   No lattice point lies between consecutive shifted lower
   envelopes.
   
   \item Hence, every lattice point in the rind lies on exactly
   one of the shifted lower envelopes.
\end{itemize}

\begin{prop}
\label{prop:RindAndStrata}
   Suppose that $\P \subset \R^{n}$ is a rational polytope with
   $\pt{0} \in \P$, and let $d(\P) \deftobe \den(\dual \P)$.
   Define the shifted cones $\cone^{i} \P$ for $0 \le i \le d(\P)$
   as above.  Then we have the following:
   \begin{enumerate}[{\rm (a)}]
      \item 
      $\llenv \cone \P = \bigsetsum_{i=0}^{d(\P) - 1}
      \parens{(\lenv \cone^{i} \P)_{\Z} - \tfrac{i}{d(\P)}\enp} ,$

      \item  
      $(\lenv \cone^{i} \P)_{\Z} = (\cone^{i} \P)_{\Z} \setminus
      (\cone^{i+1} \P)_{\Z}$ for $0 \le i \le d(\P) - 1 \, ,$
   
      \item
      $\sigma_{\lenv \cone^{i} \P} = \sigma_{\cone^{i} \P} -
      \sigma_{\cone^{i+1} \P}$ for $0 \le i \le d(\P) - 1 \, ,$
      
      \item 
      $(1 - z_{n+1})\,\sigma_{\cone \P}(\mul{z}) =
      \sum_{i=0}^{d(\P) - 1} \sigma_{\lenv \cone^{i}
      \P}(\mul{z})$.
   \end{enumerate}
\end{prop}

\begin{figure}[tbp]
   \centering
   \includegraphics{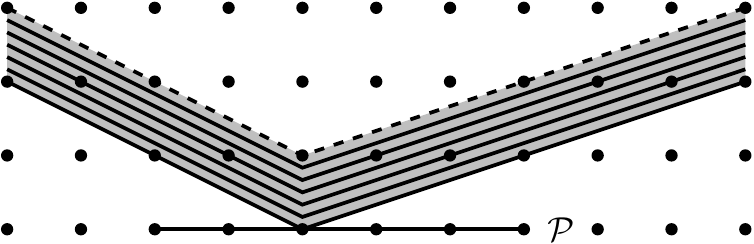}
   \caption{The shifted lower envelopes $\lenv \cone^{i} \P$, $0
   \le i \le 5$, where $\P = [-2,3] \subset \R^{1}$.  The shaded
   region is the rind of $\cone\P$.  Observe that every lattice
   point in the rind lies on one of the shifted lower envelopes
   shown.}
   %
   \label{fig:RindAndStrata}
\end{figure}

%
%

\begin{proof}
   The right-hand side of part (a) is contained in the left-hand
   side because elements of the right-hand side are points in
   $\lenv \cone \P$ that are directly beneath lattice points.  To
   see that the left-hand side of part (a) is contained in the
   right-hand side, let $\pt{x} \in \llenv \cone \P$ be given.  It
   suffices to show that $\x + \frac{i}{d(\P)}\enp \in \Z^{n+1}$
   for some $i \in \setof{0, \dotsc, d(\P) - 1}$.  Let
   \[
   \lambda
   \deftobe \max \setof{\phi(\pi(\pt{x})) \st \text{$\phi$ is a
   vertex of $\dual{\P}$}} ,
   \]
   and let $k \deftobe \ceil{\lambda}$.  Thus, $\pi(\x) \in
   \lambda\P$, but $\pi(\x) \notin (\lambda - \epsilon)\P$ for all
   $0 < \epsilon < \lambda$.  Hence, $(\pi(\x), \lambda) \in \lenv
   \cone \P$, so $\x = (\pi(\x), \lambda)$.  Now, every vertex
   $\phi$ of $\dual{\P}$ satisfies $\smash{\phi(\a) \in
   \frac{1}{d(\P)}\Z}$ for all $\a \in \Z^{n}$.  Since $\pi(\x)
   \in \Z^{n}$, we thus have that $\smash{k = \lambda +
   \frac{i}{d(\P)}}$ for some $i \in \setof{0, \dotsc, d(\P) -
   1}$.  Therefore, $\smash{\x + \frac{i}{d(\P)}\enp = (\pi(\x),
   k) \in \Z^{n+1}}$, as required.
   
   To see that the union in part (a) is disjoint, suppose that
   \begin{equation*}
      \x_{1} - \tfrac{i}{d(\P)}\enp = \x_{2} - \tfrac{j}{d(\P)} \enp
   \end{equation*}
   for some $\x_{1} \in (\lenv \cone^{i}\P)_{\Z}$ and $\x_{2} \in
   (\lenv \cone^{j}\P)_{\Z}$, where, without loss of generality,
   $0 \le i \le j < d(\P)$.  Then $\x_{2} - \x_{1} =
   \smash{
   \big(
   \frac{j}{d(\P)} - \frac{i}{d(\P)}
   \big)\, %
   \enp}$ is a lattice point and $\smash{0 \le \frac{j}{d(\P)} -
   \frac{i}{d(\P)} < 1}$.  This implies that $i = j$, showing
   disjointness and proving part~(a).
   
   To prove part (b), suppose that there is an element $\x$ on the
   right-hand side that is not on the left-hand side.  Then, for
   some integer $i$ such that $0 \le i \le d(\P) - 1$ and some
   $\lambda$ such that $\smash{\frac{i}{d(\P)} < \lambda <
   \frac{i+1}{d(\P)}}$, we have $\x - \lambda \enp \in \llenv
   \cone \P$.  Hence, by part (a), there exist $\pt{y} \in
   \Z^{n+1}$ and $j \in \setof{0, \dotsc, d(\P) - 1}$ such that
   $\smash{\pt{y} - \frac{j}{d(\P)}\enp = \x - \lambda \enp}$.
   This implies that $\smash{\lambda - \frac{j}{d(\P)}}$ is an
   integer, which is a contradiction.  This proves part (b).
   
   Part (c) follows immediately, since it is a restatement of part
   (b) in terms of generating functions.  Part (d) results from
   summing both sides of part (c) over all integers $i$ such that
   $0 \le i \le d(\P) - 1$.
\end{proof}


Using the previous proposition, we can write down a version of
Proposition \ref{prop:LatticeFreeSumIntoDisjointBoundaryandCone}
in which the sets in the disjoint union are indexed by lattice
points.  This allows us to translate the resulting set equality
directly into an equality of generating functions.

\begin{thm}
   \label{thm:NonBraunFormulaCase}
   Suppose that $\P \subset \R^{n}$ is a rational polytope and $\K
   \subset \R^{n}$ is a convex set such that $\P \oplus \K$ is a
   free sum.  Then
   \begin{equation*}
      \cone(\P \oplus \K)_{\Z} = \bigsetsum_{i = 0}^{d(\P)-1}
      \bigsetsum_{\pt{x} \in (\lenv \cone^{i} \P)_{\Z}} (\pt{x} +
      \cone_{i} \K)_{\Z} \, .
   \end{equation*}
   Therefore,
   \begin{align}
       \sigma_{\cone(\P \oplus \K)} 
        &= \sum_{i=0}^{d(\P) - 1}
        \sigma_{\lenv \cone^{i} \P} \, \sigma_{\cone_{i}\K} 
        \label{eq:NonBraunFormulaCaseLenvs} \\
        &= \sum_{i=0}^{d(\P)-1} (\sigma_{\cone^{i} \P} -
        \sigma_{\cone^{i+1} \P}) \, \sigma_{\cone_{i} \K} \, . 
        \label{eq:NonBraunFormulaCaseCones}
   \end{align}
\end{thm}

\begin{proof}
   By Proposition
   ~\ref{prop:LatticeFreeSumIntoDisjointBoundaryandCone}, 
   \begin{equation*}
      \cone(\P \oplus \K)_{\Z} = \bigsetsum_{\pt{x} \in \llenv
      \cone \P} (\pt{x} + \cone \K)_{\Z} \, .
   \end{equation*}
   By Proposition ~\ref{prop:RindAndStrata}(a), this becomes
   \begin{align*}
      \cone(\P \oplus \K)_{\Z} &= \bigsetsum_{i=0}^{d(\P)-1}
      \bigsetsum_{\pt{x} \in (\lenv \cone^{i} \P)_{\Z}} (\pt{x} -
      \tfrac{i}{d(\P)} \enp + \cone \K)_{\Z} \\
       &= \bigsetsum_{i=0}^{d(\P)-1} \bigsetsum_{\pt{x} \in (\lenv
       \cone^{i} \P)_{\Z}} (\pt{x} + \cone_{i} \K)_{\Z} \, .
   \end{align*}
   Equation \eqref{eq:NonBraunFormulaCaseLenvs} is the restatement
   of this equality in terms of generating functions, and equation
   \eqref{eq:NonBraunFormulaCaseCones} follows from Proposition
   ~\ref{prop:RindAndStrata}(c).
\end{proof}

\begin{rem}
   Some of the terms in equation
   \eqref{eq:NonBraunFormulaCaseLenvs} may be zero.  For example,
   if $\P$ is the interval $[-2, 3]$, then $\sigma_{\lenv
   \cone^{1} \P} = \sigma_{\lenv \cone^{5} \P} = 0$.  (See Figure
   \ref{fig:RindAndStrata}.)  Nonetheless, if $d(\P) > 1$, then
   $\sigma_{\lenv \cone^{i} \P} \ne 0$ for some $i \in
   \setof{1,\dotsc,d(\P)-1}$ by Proposition ~\ref{ratflexiveprop}.
\end{rem}

Before proving Theorem ~\ref{thm:ConverseBraunFormula}, we need two
lemmas constraining when lattice points can appear in the shifted
lower envelopes of cones over compact convex sets.

\begin{lem}
   \label{lem:LowerEnvelopesAreSymmetric}
   Let $\J \subset \R^{n}$ be a compact convex set, and let $\rho$
   be a rational number.  Then $(\lenv \cone \J + \rho\,
   \enp)_{\Z} \ne \emptyset$ if and only if $(\lenv \cone \J -
   \rho \, \enp)_{\Z} \ne \emptyset$.
\end{lem}

\begin{proof}
   Since $\rho \in \mathbb{Q}$, a ray $R$ originating at $\rho\,
   \enp$ contains a lattice point if and only if the inversion of
   $R$ through $\rho \, \enp$ also contains a lattice point.
   Hence, the set $\lenv \cone \J + \rho\, \enp$, which is a union
   of rays originating at $\rho\,\enp$, contains a lattice point
   if and only if its inversion through $\rho\, \enp$ contains a
   lattice point.  But $\lenv \cone \J - \rho \, \enp$ is just the
   inversion of this latter set through the origin.  That is,
   \begin{equation*}
      \lenv \cone \J - \rho\, \enp = -(-((\lenv \cone \J + \rho\,
      \enp) - \rho\, \enp) + \rho\, \enp).
   \end{equation*}
   Since inversion through the origin is a lattice-preserving
   operation, the claim follows.
%
\end{proof}


\begin{lem}\label{lem:LowerEnvelopesContainingLatticePointsHaveRationalVertices}
   Let $\Q$ be a rational polytope, and let $\rho$ be a real
   number.  If
   \begin{equation*}
      (\lenv \cone \Q + \rho\, \enp)_{\Z} \ne \emptyset,
   \end{equation*}
   then $\rho$ is a rational number.
\end{lem}
\begin{proof}
   Let $\x \in (\lenv \cone \Q + \rho\, \enp)_{\Z}$, and let $F$
   be a facet of $\cone \Q$ containing \mbox{$\x - \rho \, \enp$}.
   Then the supporting hyperplane $H$ of $\cone \Q$ at $F$ is a
   rational hyperplane containing $\rho\,\enp - \x$.  Therefore,
   the translation $H + \x$ by an integer lattice point must meet
   the $\enp$-axis at a rational point.
\end{proof}

We are now ready to prove Theorem~\ref{thm:ConverseBraunFormula}.

%

\begin{proof}[Proof of Theorem ~\ref{thm:ConverseBraunFormula}
(stated on p.\ \pageref{thm:ConverseBraunFormula})]
   \label{proof:ConverseBraunFormula} The ``if'' direction follows
   immediately from Corollary ~\ref{rationalbraunmultivarcor}.  To
   prove the converse, suppose that equation
   \eqref{eq:BraunsFormulaHolds} holds but that $\dual\P$ is not a
   lattice polyhedron.  Then, by Proposition
   ~\ref{ratflexiveprop}, $\llenv \cone \P \ne (\lenv \cone
   \P)_{\Z}$.  Hence, by Proposition ~\ref{prop:RindAndStrata}(a),
   there exists a maximum integer $j$ with $1 \le j \le d(\P) - 1$
   such that $(\lenv \cone^{j} \P)_{\Z} \ne \emptyset$.
   
   We claim that the nonemptiness of $(\lenv \cone^{j} \P)_{\Z}$,
   in combination with equation~\eqref{eq:BraunsFormulaHolds},
   implies that
   \begin{equation}
      \label{eq:3}
      (\cone_{j} \Q)_{\Z} \setminus (\cone \Q)_{\Z} = \emptyset.
   \end{equation}
   To see this, apply Proposition \ref{prop:RindAndStrata}(d) to
   rewrite equation~\eqref{eq:BraunsFormulaHolds} as follows:
   \begin{equation*}
      \sigma_{\cone(\P \oplus \Q)} = \sum_{i=0}^{d(\P) - 1}
      \sigma_{\lenv \cone^{i} \P} \, \sigma_{\cone \Q}.
   \end{equation*}
   Together with Theorem \ref{thm:NonBraunFormulaCase}, this 
   yields
   \begin{equation*}
       \sum_{i=0}^{d(\P) - 1} \sigma_{\lenv \cone^{i} \P} \,
       \sigma_{\cone_{i}\Q} = \sum_{i=0}^{d(\P) - 1} \sigma_{\lenv
       \cone^{i} \P} \, \sigma_{\cone \Q},
   \end{equation*}
   or, equivalently,
   \begin{equation*}
      \sum_{i=0}^{d(\P) -1} \sigma_{\lenv \cone^{i} \P} 
      (\sigma_{\cone_{i}\Q} - \sigma_{\cone \Q}) = 0.
   \end{equation*}
   Since $\cone \Q \subset \cone_{i} \Q$, the monomials on the
   left-hand side all have nonnegative coefficients.  In
   particular, since $(\lenv \cone^{j} \P)_{\Z} \ne \emptyset$, we
   must have that $\sigma_{\cone_{j}\Q} - \sigma_{\cone \Q} = 0$,
   proving equation \eqref{eq:3}.

   
   We now show that $j/d(\P) \ge \frac{1}{2}$.  The maximality of
   $j$ implies that
   \begin{equation*}
      \bigsetsum_{i=j+1}^{d(\P)-1} (\lenv \cone^{i}\P)_{\Z} =
      \emptyset,
   \end{equation*}
   which, by Lemma~\ref{lem:LowerEnvelopesAreSymmetric}, becomes
   \begin{equation*}
      \bigsetsum_{i=j+1}^{d(\P)-1} (\lenv \cone_{i}\P)_{\Z} =
      \emptyset.
   \end{equation*}
   Translating by $\enp$ and then reversing the order of the
   disjoint union yields
   \begin{equation*}
      \bigsetsum_{i=j+1}^{d(\P)-1} (\lenv \cone^{d(\P) -
      i}\P)_{\Z} \, = \, \bigsetsum_{i=1}^{d(\P) - j - 1}
      (\lenv \cone^{i}\P)_{\Z} = \emptyset.
   \end{equation*}
   Since $(\lenv \cone^{j} \P)_{\Z} \ne \emptyset$ and $j \ge 1$,
   we must have $j > d(\P) - j - 1$, or $j/d(\P) \ge \frac{1}{2}$,
   as claimed.
   
   We now apply similar reasoning to $\cone \Q$.  Equation
   \eqref{eq:3} implies that
   \begin{equation}
      \label{eq:4}
      \bigsetsum_{\substack{\rho \in \mathbb{Q} \st \\ 0 < \rho \le
      j/d(\P)}}(\lenv \cone \Q - \rho\, \enp)_{\Z} = \emptyset.
   \end{equation}
   Once again applying Lemma~\ref{lem:LowerEnvelopesAreSymmetric},
   we get
   \begin{equation}
      \label{eq:5}
      \bigsetsum_{\substack{\rho \in \mathbb{Q} \st \\ 0 < \rho \le
      j/d(\P)}}(\lenv \cone \Q + \rho\, \enp)_{\Z} = \emptyset,
   \end{equation}
   while, translating the sets in equation \eqref{eq:4} by $\enp$
   and then reversing the order of the disjoint union, we have
   \begin{equation}
      \label{eq:6}
      \bigsetsum_{\substack{\rho \in \mathbb{Q} \st \\ 0 < \rho
      \le j/d(\P)}}(\lenv \cone \Q + (1 -\rho) \enp)_{\Z} =
      \bigsetsum_{\substack{\rho \in \mathbb{Q} \st \\ 1-j/d(\P)
      \le \rho < 1}}(\lenv \cone \Q + \rho \, \enp)_{\Z} =
      \emptyset.
   \end{equation}
   Since $j/d(\P) \ge \frac{1}{2}$, we can combine equalities
   \eqref{eq:5} and \eqref{eq:6} to conclude that
   \begin{equation*}
      \bigsetsum_{\substack{\rho \in \mathbb{Q} \st \\ 0 < \rho <
      1}}(\lenv \cone \Q + \rho \, \enp)_{\Z} = \emptyset.
   \end{equation*}
   Hence, by
   Lemma~\ref{lem:LowerEnvelopesContainingLatticePointsHaveRationalVertices},
   \begin{equation*}
      (\cone \Q)_{\Z} \setminus (\cone \Q + \enp)_{\Z} = (\lenv
      \cone \Q)_{\Z}.
   \end{equation*}
   Thus, by Proposition~\ref{settogeneratingfctprop}, we have that
   $\llenv \cone \Q = (\lenv \cone \Q)_{\Z}$.  Therefore, by
   Proposition~\ref{ratflexiveprop}, $\dual \Q$ is a lattice
   polyhedron.
\end{proof}

It is now straightforward to prove
Theorem~\ref{thm:ConverseBraunUnivariate}, the univariate analogue
of Theorem~\ref{thm:ConverseBraunFormula}.

\begin{proof}[Proof of Theorem \ref{thm:ConverseBraunUnivariate}
(stated on p.\ \pageref{thm:ConverseBraunUnivariate})] %
   \label{proof:ConverseBraunUnivariate}
   Corollary \ref{braunthmgeneralizedcor} already established
   that, if either $\dual \P$ or $\dual \Q$ is a lattice
   polyhedron, then equations \eqref{eq:ConverseBraunEhrhart} and
   \eqref{eq:ConverseBraunDelta} hold.  To prove the converse
   suppose that neither $\dual \P$ nor $\dual \Q$ is a lattice
   polyhedron.  Then, by Theorem \ref{thm:ConverseBraunFormula},
   \begin{equation*}
      \sigma_{\cone(\P \oplus \Q)}(\mul{z}) \ne
      (1-z_{n+1})\,\sigma_{\cone
      \P}(\mul{z})\,\sigma_{\cone\Q}(\mul{z}).
   \end{equation*}
   By Theorem \ref{thm:NonBraunFormulaCase} and Proposition
   \ref{prop:RindAndStrata}(d), this becomes
   \begin{equation*}
      \sum_{i=0}^{d(\P) - 1} \sigma_{\lenv \cone^{i} \P} \,
      \sigma_{\cone_{i}\Q} \ne \sum_{i=0}^{d(\P) - 1}
      \sigma_{\lenv \cone^{i} \P} \, \sigma_{\cone \Q}.
   \end{equation*}
   Now, since $\cone \Q \subset \cone_{i} \Q$ for all $i$, every
   monomial on the right-hand side appears on the left-hand side.
   Thus, 
   \begin{equation}\label{eq:ConverseBraunDelta1}
      \sigma_{\cone(\P \oplus \Q)}(\mul{z}) =
      (1-z_{n+1})\,\sigma_{\cone
      \P}(\mul{z})\,\sigma_{\cone\Q}(\mul{z}) + \tau(\mul{z})
   \end{equation}
   for some nonzero Laurent series $\tau(\mul{z})$ with
   nonnegative coefficients.
   Hence, specializing equation \eqref{eq:ConverseBraunDelta1} at
   $\mul{z} = (1, \dotsc, 1, t)$ yields
   \begin{equation*}
      \Ehr_{\P \oplus \Q}(t) = (1 - t) \Ehr_{\P}(t) \Ehr_{\Q}(t) 
      + F(t),
   \end{equation*}
   where $F(t)$ is a nonzero power series.  In particular,
   equation \eqref{eq:ConverseBraunEhrhart} does not hold.
   Multiplying through by the denominator of the rational function
   $\Ehr_{\P \oplus \Q}(t)$ shows that equation
   \eqref{eq:ConverseBraunDelta} also does not hold.  (Equation
   \eqref{eq:ConverseBraunDeltaLattice} is just the case of
   equation \eqref{eq:ConverseBraunDelta} in which $\den(\P) =
   \den(\Q) = 1$.)
\end{proof}

\section{Sums of polytopes intersecting at rational points}
\label{sec:nonlatticepoints}

In previous sections, we considered the generating function
$\sigma_{\cone(\J \join \K)}$ where $\J \join \K$ was a
free sum.  In particular, $\J$ and $\K$ intersected only at the
origin.  Matters are essentially the same if $\J$ and $\K$
intersect at an arbitrary \emph{lattice} point $\p$ in $\Z^{n}$,
since we can reduce the computation of $\sigma_{\cone(\J \join
\K)}$ to the previous case via the equation
\begin{equation*}
   \sigma_{\cone(\J \join \K)}(\mul{z}) = \sigma_{\cone((\J - 
   \p)\oplus(\K-\p))}(z_{1},\dotsc,z_{n},\mul{z}^{\alpha(\p)}).
\end{equation*}
(Here, in accordance with the convention mentioned in Section
~\ref{sec:Introduction}, $\mul{z}^{\alpha(\p)}$ denotes the
monomial $z_{1}^{p_{1}}\dotsm z_{n}^{p_{n}}z^{\null}_{n+1}$, where
$\p = (p_{1},\dotsc,p_{n})$.)

We now turn to the case where $\J$ and $\K$ intersect in an
arbitrary rational point in $\mathbb{Q}^n$.  Our results in this
section generalize the propositions in Section
~\ref{sec:FreeSumDecompositions} and some of the results in Section
~\ref{sec:BraunFormula}.  We begin by extending our earlier
definitions of lower (lattice) envelopes to accommodate projections
that are not in the vertical direction.

Given $\p \in \mathbb{Q}^{n}$, define $\pi^\p \maps \R^{n+1} \to
\R^{n}$ via $\pi^\p(\pt{x})=\pi(\pt{x}-x_{n+1}\alpha(\p))$ where
$x_{n+1}$ is the last coordinate of $\pt{x}$.  Thus, instead of
projecting vertically down to $\R^{n}$ (as in previous sections),
$\pi^\p$ projects parallel to $\alpha(\p)$.  Note that
$\pi=\pi^{\pt{0}}$.  However, in general we may not have
$\pi^\p(\Z^{n+1}) = \Z^{n}$.

Given a closed linear cone $\C \subset \R^{n+1}$ not containing
$-\alpha(\p)$, define $\epsilon_\C^\p \maps \C \to \bnd \C$ via
\[\epsilon_\C^\p(\pt{x}) \deftobe \x - \max\setof{\lambda \in \R
\st \x - \lambda \alpha(\p) \in \C} \alpha(\p).\] Given a compact
convex set $\J\subset \R^{n}$, we will write $\epsilon^{\p}_{\J}$
as an abbreviation for $\epsilon^{\p}_{\cone \J}$.  We then define
the \emph{$\p$-lower envelope} $\plenv \C$ of $\C$ via
\begin{equation*}
   \plenv \C \deftobe \epsilon_\C^\p(\C) \, .
\end{equation*}
Similar to the lower envelope, the $\p$-lower envelope of $\C$ is
the set of points in $\C$ that are ``minimal in the direction of
$\alpha(\p)$''.  Finally, we introduce the notion of
\emph{$\p$-lower lattice envelope} of $\C$, defined as
\begin{equation*}
   \pllenv \C \deftobe \epsilon_\C^\p(\C_{\Z}) \, .
\end{equation*}
Thus the $\p$-lower lattice envelope is the projection of the
lattice points in $\C$ in the direction parallel to $\alpha(\p)$
onto the $\p$-lower envelope of $\C$.  The lower (lattice)
envelope of previous sections reappears as the special case $\p =
\pt{0}$.


We are now ready to state the generalizations of the propositions
from Section ~\ref{sec:FreeSumDecompositions}.

%

\begin{prop}\label{prop:somewhatintermediateaffine}
   Suppose that $\J, \K \subset \R^{n}$ are convex sets with $\J$
   compact.  Suppose in addition that the affine spans of $\J$ and
   $\K$ intersect in exactly one rational point $\p \in \K$.  Then
   \begin{equation}
      \label{eq:AffineFreeSumIntoDisjointBoundaryandCone}
      \cone(\J \join \K) = \bigsetsum_{\pt{x} \in \plenv \cone \J}
      (\pt{x} + \cone \K) \, .
   \end{equation}
   \label{prop:FreeSumIntoDisjointBoundaryandConeinpcase}
\end{prop}

Once we note that, for $\pt{x} \in \cone \J$, $\pi^\p(\pt{x})$ is
in $\lin (\J-\p)$, the proof of this proposition is the same as
the proof of Proposition
~\ref{prop:FreeSumIntoDisjointBoundaryandCone} with the
appropriate replacements (such as $\pi$ replaced by $\pi^\p$, and
$\epsilon_\J$ replaced by $\epsilon_\J^\p$).

We now seek a restriction of equation
\eqref{eq:AffineFreeSumIntoDisjointBoundaryandCone} to lattice
points that is in the spirit of Proposition
~\ref{prop:LatticeFreeSumIntoDisjointBoundaryandCone}.  To this
end, we define an analogue of the free-sum operation, which we
call an \emph{affine free sum}.  Recall that, for $\J \join \K$ to
be a free sum, we required that $(\lin \J)_{\Z}$ and $(\lin
\K)_{\Z}$ be complementary sublattices of $\Z^{n}$.  One
complication of our present case is that $\pi^\p(\pt{x})$ is not
necessarily a lattice point for every lattice point $\pt{x}$ in
$\cone \J$.  Thus, we consider the refinement $\Lambda^{\p}
\deftobe \pi^\p(\Z^{n+1})$ of $\Z^{n}$.  There are several
equivalent characterizations of this lattice:
\begin{enumerate}
   \item 
   $\Lambda^{\p} = \pi^{\p}(\Z^{n+1})$.
   
   \item 
   $\Lambda^\p$ is the lattice in $\R^n$ generated by
   $\setof{\pt{e}_{1}, \dotsc, \pt{e}_{n}, \p}$ under integer
   linear combinations.
   
   \item 
   $\Lambda^{\p} =\bigsetsum_{k=0}^{r-1} \parens{\Z^n-k\p}$, where
   $r \deftobe \den(\p)$ is the least common multiple of the
   denominators of the coordinates of $\p$.
\end{enumerate}
We adapt our earlier notation and terminology to work with the
lattice $\Lambda^\p$ as follows.  For a subset $S$ of $\R^n$, let
$S_{\Lambda^\p}$ denote the set of points in $S \cap \Lambda^\p$.
We say that two sublattices $\L, \M \subset \Lambda^\p$ are
\emph{complementary sublattices of $\Lambda^\p$} if each element
of $(\lin(\L \cup \M))_{\Lambda^\p}$ is the sum of a unique
element of $\L$ and a unique element of $\M$.



Given convex sets $\J$ and $\K$ in $\R^n$, we call $\J \join \K$
an \emph{affine free sum} if $\J$ and $\K$ intersect at a point
$\p \in \mathbb{Q}^n$ such that $(\lin(\J-\p))_{\Lambda^\p}$ and
$(\lin(\K-\p))_{\Lambda^\p}$ are complementary sublattices of
$\Lambda^\p$.  Equivalently, $\J \oplus \K$ is an affine free sum
if $\J$ and $\K$ intersect at a unique rational point and 
\begin{equation*}
   \lin(\cone(\J \oplus \K))_{\Z} = \lin(\cone \J)_{\Z} +
   \lin(\cone \K)_{\Z},
\end{equation*}
where the sum on the right is the Minkowski sum.

\begin{prop}
   Suppose that $\J, \K \subset \R^{n}$ are convex sets such that
   $\J$ is compact and $\J \join \K$ is an affine free sum of
   convex sets intersecting at $\p \in \mathbb{Q}^n$.
   Then
   \begin{equation*}
      \cone(\J \join \K)_{\Z} = \bigsetsum_{\pt{x} \in \pllenv
      \cone \J} (\pt{x} + \cone \K)_{\Z} \, .
   \end{equation*}
\label{prop:LatticeFreeSumIntoDisjointBoundaryandConeinpcase}
\end{prop}

\begin{proof}
   Elements on the right-hand side are integer lattice points that
   are contained in $\cone(\J \join \K)$ by
   Proposition ~\ref{prop:somewhatintermediateaffine}.  Hence, such elements are in the left-hand side.
   
   To prove the converse containment, let $\pt{w} \in \cone(\J
   \join \K)_{\Z}$ be given.  Then by the previous proposition,
   $\pt{w}=\pt{x}+\pt{y}$ where $\pt{x} \in \plenv \cone \J$ and
   $\pt{y} \in \cone \K$.  Thus, $\pi^\p(\pt{w}) = \pi^\p(\pt{x})
   + \pi^\p(\pt{y})$.  Now, $\pi^\p(\pt{w})$ is in $\lin((\J \cup
   \K)-\p)_{\Lambda^\p}$, while $\pi^\p(\pt{x}) \in \lin (\J-\p)$
   and $\pi^\p(\pt{y}) \in \lin (\K-\p)$.  Thus, the
   complementarity of $(\lin (\J-\p))_{\Lambda^\p}$ and $(\lin
   (\K-\p))_{\Lambda^\p}$ implies that $\pi^\p(\pt{x})$ is in
   $\Lambda^\p$.  Hence there exists a non-negative integer
   $\lambda$ such that $(\pi^\p(\pt{x}),0)+\lambda
   \alpha(\p)=(\pi^\p(\pt{x})+\lambda\p,\lambda)$ is an integer
   lattice point in $\cone \J$.  Since
   $\epsilon_\J^\p((\pi^\p(\pt{x})+\lambda\p,\lambda))=\pt{x}$, we
   have $\pt{x} \in \epsilon^\p_\J((\cone \J)_{\Z})$, and the
   result follows.
\end{proof}

We now turn to the rational generating function $\sigma_{\cone (\J
\join \K)}$ and state the generalizations of Proposition
~\ref{settogeneratingfctprop} and Theorem
~\ref{thm:MultivariateBraunsFormula}.

\begin{prop}\label{prop:settogeneratingfctinpcase}
   Fix a compact convex set $\J \subset \R^{n}$ containing
   $\pt{p}\in \mathbb{Q}^n$.  Let $r \deftobe \den(\p)$.  Then the
   following are equivalent:
   \begin{enumerate}[{\rm (a)}]
      \item  
      $\pllenv \cone \J = (\plenv \cone \J)_{\Z}$,
   
      \item  
      $(\plenv \cone \J)_{\Z} = (\cone \J)_{\Z} \setminus (\cone \J
      + r\alpha(\p))_{\Z}$,
   
      \item  
      $\sigma_{\plenv \cone \J}(\mul{z}) = (1 - \mul{z}^{r\alpha(\p)}) \,
      \sigma_{\cone \J}(\mul{z})$.
   \end{enumerate}
\end{prop}

\begin{proof}
   We first show that (a) and (b) are equivalent.  By definition
   of the $\p$-lower envelope and $\p$-lower lattice envelope, we
   have $(\plenv \cone \J)_{\Z} \subset \pllenv \cone \J$ and
   $(\plenv \cone \J)_{\Z} \subset (\cone \J)_{\Z} \setminus
   (\cone \J+ r\alpha(\p))_{\Z}$.  As in the proof of Proposition
   ~\ref{settogeneratingfctprop}, we observe that $\x \mapsto
   \epsilon_{\J}^\p(\x)$ is a bijection between non-lower-envelope
   points in $(\cone \J)_{\Z} \setminus (\cone \J+
   r\alpha(\p))_{\Z}$ and non-lattice points in $\pllenv \cone
   \J$, with the inverse given by $\x \mapsto \x +
   \min\setof{\lambda \in \R \st \x + \lambda \alpha(\p) \in
   (\cone \J)_{\Z}} \alpha(\p)$.  The rest of the proof is the
   same as the proof of Proposition ~\ref{settogeneratingfctprop}.
\end{proof}

\begin{thm}
   \label{thm:MultivariateBraunsFormulainpcase}
   Suppose that $\J, \K \subset \R^{n}$ are convex sets such that
   $\J$ is compact and $\J \join \K$ is an affine free sum of
   convex sets intersecting at $\p \in \mathbb{Q}^n$.  Further
   suppose that $\pllenv \cone \J = (\plenv \cone \J)_{\Z}$.  Then
   \begin{equation}
      \label{eq:AffineBraunFormula}
      \sigma_{\cone (\J \join \K)}(\mul{z}) = ( 1 -
      \mul{z}^{r\alpha(\p)} ) \, \sigma_{\cone \J}(\mul{z}) \,
      \sigma_{\cone \K}(\mul{z}) \, ,
   \end{equation}
   where $r \deftobe \den(\p)$.
\end{thm}

The proof is the same as the proof of Theorem
~\ref{thm:MultivariateBraunsFormula} with the appropriate
replacements.

\begin{rem}
   \label{rem:Bruns}
   It is straightforward to adapt the arguments in Section
   \ref{NonBraunFormulaCase} to prove a converse of Theorem
   \ref{thm:MultivariateBraunsFormulainpcase} analogous to Theorem
   \ref{thm:ConverseBraunFormula}.  That is, one can show that, if
   $\P \join \Q$ is an affine free sum of rational polytopes
   intersecting at $\p \in \mathbb{Q}^n$, and
   \begin{equation*}
      \sigma_{\cone (\P \join \Q)}(\mul{z}) = ( 1 -
      \mul{z}^{\den(\p) \alpha(\p)} ) \, \sigma_{\cone \P}(\mul{z}) \,
      \sigma_{\cone \Q}(\mul{z}) \, ,
   \end{equation*}
   then $\pllenv \cone \P = (\plenv \cone \P)_{\Z}$.  A recent
   preprint of W.~Bruns proves this result in the general context
   of arbitrary affine monoids~\cite{Bruns2013}.  As is the case
   with Theorem \ref{thm:MultivariateBraunsFormula}, whether the
   converse of Theorem \ref{thm:MultivariateBraunsFormulainpcase}
   holds for free sums $\J \oplus \K$ of arbitrary convex sets is
   still an open question.
\end{rem}

\begin{exm}
   \label{exm:AffineFreeSum1}
   Let $\J$ be the line segment from $(0,0)$ to $(1,0)$ in $\R^2$
   and let $\K$ be the line segment from $\parens{\frac{1}{2},-1}$
   to $\parens{\frac{1}{2},1}$ in $\R^2$.  Then $\J$ and $\K$
   intersect at $\p \deftobe \parens{\frac{1}{2},0}$, and $\J
   \join \K$ is an affine free sum.  The $\p$-lower envelope of
   $\cone \J$ is the boundary of the cone, and the set of lattice
   points in the boundary is precisely the set $(\cone \J)_{\Z}
   \setminus (\cone \J+ r\alpha(\p))_{\Z}$, where $r \deftobe
   \den(\p) = 2$.  Thus, $\J$ satisfies the conditions in
   Proposition ~\ref{prop:settogeneratingfctinpcase}.  Hence,
   Theorem ~\ref{thm:MultivariateBraunsFormulainpcase} applies,
   yielding
   \begin{align*}
      \sigma_{\cone (\J \join \K)}(z_1, z_2, z_3) 
       &= (1 - z_1 z_3^2) \, \sigma_{\cone \J}(z_1, z_2, z_3) \, \sigma_{\cone \K}(z_1, z_2, z_3) \, \\
       &= (1 - z_1 z_3^2) \, \frac{1}{(1 - z_{3})(1-z_{1}z_{3})} 
          \frac{1 + z_{1}z_{2}^{-1}z_{3}^{2} + z_{1}z_{3}^{2} + z_{1}z_{2}z_{3}^{2}}{(1 -
z_{1}z_{2}^{-2}z_{3}^{2})(1 - z_{1} z_{2}^{2}z_{3}^{2})}.
   \end{align*}
\end{exm}

\begin{exm}
   Theorem ~\ref{thm:MultivariateBraunsFormulainpcase} need not
   hold if we drop the condition that $\pllenv \cone \J = (\plenv
   \cone \J)_{\Z}$.  If we keep $\J$ the same set as in Example
   ~\ref{exm:AffineFreeSum1}, but let $\K$ be the line segment from
   $\parens{\frac{1}{3},-1}$ to $\parens{\frac{1}{3},1}$ in
   $\R^2$, then $\p=\parens{\frac{1}{3},0}$,
   $\alpha(\p)=\parens{\frac{1}{3},0,1}$ and $r=3$.  The
   $\p$-lower envelope of $\cone \J$ is still the boundary of the
   cone, but there are now lattice points in the set $(\cone
   \J)_{\Z} \setminus (\cone \J+ r\alpha(\p))_{\Z}$ that are not
   in the boundary of the cone.  Thus, the conditions in
   Proposition ~\ref{prop:settogeneratingfctinpcase} are not true
   of $\J$, and so we would need to use generalizations of results
   from Section ~\ref{NonBraunFormulaCase} to compute
   $\sigma_{\cone (\J \join \K)}(\mul{z})$.  We do not develop
   such generalizations here.
\end{exm}

We now consider an important class of polytopes for which the
conditions in Proposition ~\ref{prop:settogeneratingfctinpcase}
are true, so that Theorem
~\ref{thm:MultivariateBraunsFormulainpcase} applies when one of
the summands is a polytope from this class.  A lattice polytope
$\P$ is \emph{Gorenstein of index $k$} if there exists a lattice
point $\pt{m}$ such that $k\P-\pt{m}$ is a reflexive polytope.  In
particular, $\pt{m}$ is the unique lattice point in $k \intr{\P}$.
The recent paper ~\cite{Nill:2010fk} discusses Braun's formula in
the context of Gorenstein polytopes and nef-partitions.
   
\begin{prop}
   \label{prop:GorensteinLowerLatticePoints}
   Suppose that $\P\subset \R^n$ is a Gorenstein polytope of index
   $k$.  Let $\pt{m}$ be the unique lattice point in $\intr{k\P}$,
   and let $\p := \frac 1 k \pt{m}$.
   Then $\pllenv \cone \P = (\plenv \cone \P)_{\Z}$.
\end{prop}

\begin{proof}
   Since $\p \in \intr\P$, we have that $\plenv \cone \P = \bnd
   \cone \P$.  It is well known that the Gorenstein property
   implies that $\intr{(\cone \P)_{\Z}} = (\cone \P +
   k\alpha(\p))_{\Z}$ (see, \emph{e.g.}, ~\cite{BR2007}).  In
   particular, $k = \den(\p)$.  The result follows from
   Proposition \ref{prop:settogeneratingfctinpcase}(b).
\end{proof}

\begin{cor}\label{cor:GorensteinBraunCor}
   Suppose that $\P\subset \R^n$ is a Gorenstein polytope of index
   $k$.  Let $\pt{m}$ be the unique lattice point in $\intr{k\P}$,
   and let $\p := \frac 1 k \pt{m}$.  Let $\K \subset \R^{n}$ be a
   convex set containing $\p$ such that $\P \join \K$ is an affine
   free sum.
   Then
   \begin{equation*}
      \sigma_{\cone (\P \join \K)}(\mul{z}) = ( 1 -
      \mul{z}^{k\alpha(\p)}) \, \sigma_{\cone \P}(\mul{z}) \,
      \sigma_{\cone \K}(\mul{z}) \, .
   \end{equation*}
\end{cor}

Example ~\ref{exm:AffineFreeSum1} is an instance of this corollary,
as the line segment $\J$ in that example is a Gorenstein polytope
of index 2 with $\pt{m}=(1,0)$.

In Section ~\ref{sec:BraunFormula}, we noted that the conditions
in Proposition ~\ref{settogeneratingfctprop} applied to a broader
family than just the reflexive polytopes.  Indeed, in that
context, the integrality of the vertices of the polytope was
unimportant; all that we needed was that the polar dual be a
lattice polyhedron (cf.\ Proposition~\ref{ratflexiveprop}).  It is
natural to expect that the Gorenstein condition in Proposition
~\ref{prop:GorensteinLowerLatticePoints} can similarly be weakened
to admit non-lattice polytopes.  For example, one might hope that,
in Proposition~\ref{prop:GorensteinLowerLatticePoints}, we could
take $\P$ to be any rational polytope such that, for some integer
$k$ and some lattice point $\pt{m} \in k\P$, $\dual{(k\P -
\pt{m})}$ is a lattice polyhedron.  Unfortunately, this is not the
case in general, as the following example shows.

\begin{exm}
   Let $\P \deftobe [\frac{1}{4}, \frac{3}{4}] \subset \R^1$.
   Observe that $2\P = [\frac{1}{2}, \frac{3}{2}]$ contains the
   lattice point $\pt{m} \deftobe 1$ and that the polar dual of
   $2\P-\pt{m} = [-\frac{1}{2}, \frac{1}{2}]$ is the lattice
   polytope $[-2, 2] \subset (\R^{1})^{\ast}$.  Nonetheless,
   putting $\p \deftobe \frac{1}{2}\pt{m}$, the $\p$-lower lattice
   envelope of $\cone \P$ contains the non-lattice point
   $(\frac{1}{2}, 2)$.  Therefore, the conclusion of Proposition
   ~\ref{prop:GorensteinLowerLatticePoints} is not true of $\P$.
\end{exm}

As mentioned in Remark~\ref{rem:Bruns}, recent results by
W.~Bruns~\cite{Bruns2013} generalize our observations to the
context of general affine monoids.  Nevertheless, as the example
above shows, there still remains the problem of characterizing
when equation~\eqref{eq:AffineBraunFormula} applies in terms of
the summand polytopes, as in Proposition~\ref{ratflexiveprop},
rather than in terms of the cones over them, as
in~\cite{Bruns2013} and Remark~\ref{rem:Bruns}.

\bibliographystyle{amsplain-fi-arxlast}

\providecommand{\bysame}{\leavevmode\hbox to3em{\hrulefill}\thinspace}
\providecommand{\MR}{\relax\ifhmode\unskip\space\fi MR }
\providecommand{\MRhref}[2]{%
  \href{http://www.ams.org/mathscinet-getitem?mr=#1}{#2}
}
\providecommand{\href}[2]{#2}

\end{document}